\documentclass[12pt]{amsart}
\usepackage[colorlinks=true,citecolor=green,linkcolor=magenta]{hyperref}
\usepackage{amsmath}
\usepackage{verbatim}
\usepackage{amsfonts}
\usepackage{amssymb}
\usepackage{blkarray}
\usepackage{color,colortbl}
\usepackage[all]{xy}  
\usepackage{enumerate}
\usepackage[top=1in, bottom=1in, left=1in, right=1in]{geometry}
\usepackage{mathrsfs}

\usepackage{stmaryrd}
\usepackage{bbold}

\usepackage{cleveref}
\usepackage{soul}
\usepackage{arydshln,xcolor}
\usepackage{tikz-cd}
\usepackage{graphicx}
\def \ii{{\relax\ifmmode \mathrm{i} \else i \fi}}
\def \ee{{\relax\ifmmode \mathrm{e} \else e \fi}}
\def \dd{{\relax\ifmmode \mathrm{d} \else d \fi}}
\def \g{{\relax\ifmmode \mathrm{g} \else g \fi}}

\def \C{\mathbb C}

\def \N{\mathbb N}

\theoremstyle{plain}
\newtheorem{satz}{Theorem}[section]

\newtheorem{lemma}[satz]{Lemma} 
\newtheorem{Corollary}[satz]{Corollary}
\newtheorem{proposition}[satz]{Proposition}

\theoremstyle{definition}
\newtheorem{beispiel}{Example}[section]

\newtheorem{definition}{Definition}[section]
\newtheorem{notiz}{Remark}[section]

\renewcommand{\leq}{\leqslant}
\renewcommand{\geq}{\geqslant}

\title{Optimal Parametrizations and Valuations}

\author[H\"ubl]{Reinhold H\"ubl}
\email{Reinhold.Huebl@dhbw.de}
\address{Zentrum f\"ur mathematisch--naturwissenschaftliches Basiswissen, DHBW Mannheim, 68163 Mannheim, Germany}

\address{}
\author[Huneke]{Craig Huneke}
\email{clh4xd@virginia.edu}
\address{Department of Mathematics, University of Virginia, Charlottesville, VA 22904-4135, USA}

\author[Maitra]{Sarasij Maitra}
\email{smaitra2@haverford.edu}
\address{Department of Mathematics and Statistics, Haverford College, 370 Lancaster Avenue, Haverford, PA 19041}

\author[Mukundan]{Vivek Mukundan}
\email{vmukunda@iitd.ac.in}
\address{Department of Mathematics, Indian Institute of Technology Delhi, Hauz Khas, Delhi - 110016, India.}

\begin{document}
\begin{abstract}
    This article discusses a way for uniquely setting up the valuations for the minimal generators of the maximal ideal of a one dimensional complete reduced and irreducible local algebra over an algebraically closed field, when treated as a subring of its integral closure. Our observations are a generalization of the more well-studied case of a numerical semigroup ring. These results provide completion to some missing arguments in certain proofs present in the existing literature, including some results concerning a long-standing conjecture of R. Berger.
\end{abstract}
\maketitle
\section{Introduction}
The study of one-dimensional complete local domains has long been central to commutative algebra and singularity theory. Such rings arise naturally as localizations of curve singularities and exhibit rich interactions between their algebraic invariants and the geometry of the underlying space. Among the fundamental objects attached to such a ring $(R,\mathfrak{m})$ are the valuations of elements of its maximal ideal (when viewed as a subring of its integral closure). Many results in the literature demonstrate that valuations serve as a useful tool to extract  information about the ring $R$ (see for example \cite{Kunz70,HK,matsuoka1971degree,barucci1997maximality,d1997canonical,barucci2000analytically,oneto2007invariants,maitra2023extremal,herzog2023tiny,maitra2024two,kobayashi2025set}).

A particularly well-studied case arises when $R$ is a numerical semigroup ring, i.e., a subring of $k[[t]]$ generated by monomials $t^{w_1}, \dots, t^{w_n}$ with $\gcd(w_1,\dots,w_n)=1$. In this situation the valuation semigroup $V(R)$ of $R$ is generated by $\{w_1, \dots, w_n\}$, and the structure of $R$ can be described explicitly in terms of this information. This exchange between the algebra of $R$ and the combinatorics of its semigroup $V(R)$ has been explored in classical work of Herzog and Kunz \cite{HK} and in subsequent literature (see for instance \cite{barucci2006numerical,barucci2009propinquity,barucci2010,rosales2009irreducible,stamate2016betti,goto2018pseudo,endo2022ulrich,bhardwaj2024reduced,celikbas2024semidualizing}). However, for general one-dimensional complete local domains, where the embedding dimension and valuations are not controlled by a numerical semigroup, the picture is far less understood. In particular, although any such ring $R$ admits a parametrization of the form  $R = k[[\alpha_1 t^{a_1}, \dots, \alpha_n t^{a_n}]] \subseteq k[[t]]$, with units $\alpha_i \in k[[t]]$ and integers $a_i$, there is in general no canonical choice of the exponents $a_i$. Their algebraic significance remains subtle and much less understood, except in the case of plane curve singularities, where classical work of Zariski \cite{Zariskibook} provides a clear interpretation.

The purpose of this article is to introduce a canonical and optimal way to assign valuations to the minimal generators of $\mathfrak{m}$, thereby resolving this ambiguity. We prove that the sequence of integers $\{a_1, \dots, a_n\} = v(\mathfrak{m}) \setminus v(\mathfrak{m}^2)$, where $v$ is the $t$--adic valuation on $K = Q(R)$ possesses two key properties (by which it is determined uniquely): its cardinality equals the embedding dimension of $R$, and for any choice of elements $x_1,\dots,x_n \in R$ with $v(x_i)=a_i$, one has $R = k[[x_1,\dots,x_n]]$. We call this sequence the \emph{Herzog--Kunz sequence} of $R$, and the corresponding generators the \emph{Herzog--Kunz generators}. This characterization generalizes the classical semigroup case and strengthens the framework developed by Herzog and Kunz. In particular, it provides a new canonical description of one-dimensional complete local domains in terms of valuations, resolving longstanding ambiguities in their parametrization.

The consequences of this description are advantageous. First, the Herzog--Kunz sequence can be constructed inductively, beginning with the multiplicity of $R$. Second, it is always contained in the set of minimal generators of the value semigroup $V(R)$, and in the case of numerical semigroup rings the two coincide. Concrete examples demonstrate that the inclusion can be strict, thereby highlighting the finer structural information encoded by the Herzog--Kunz sequence. These results complete certain missing arguments in the literature and sharpen previously known theorems.

One important application of the work concerns torsion in the module of differentials. Correcting the proof on prior work \cite{HMM}, we show that if $f \in R$ has valuation greater than the largest element of the Herzog--Kunz sequence, then $f \in \mathfrak{m}^2$. This yields a simple criterion for when the conductor ideal $\mathfrak{C}_R$ is contained in $\mathfrak{m}^2$. Moreover, when this containment fails, one can explicitly construct torsion in the module of universally finite K\"ahler differentials $\widetilde{\Omega}^1_{R/k}$ (cf.\ \cite[\S 11]{KD}, ), sharpening the proof of this known result from \cite{HMM}. These arguments fill gaps in existing proofs and establish a clearer link between valuations, conductor ideals, and torsion $\tau(\widetilde{\Omega}^1_{R/k})$ of the module of differentials $\widetilde{\Omega}^1_{R/k}$.

In the final part of the paper we investigate the flexibility of parametrizations. While the Herzog--Kunz sequence furnishes an optimal valuation profile, arbitrary minimal generating sets of $\mathfrak{m}$ need not display this structure. Nevertheless, we show that such sets can always be modified in sufficiently high degrees to recover the canonical valuations. Furthermore, we prove that generators can be truncated at order $\max\{a_n+1, c_R\}$, where $c_R$ is the conductor degree, without altering the ring. In particular, every such $R$ can be generated by polynomials in the parameter $t$. This result both improves bounds available in earlier work and underscores the rigidity of the canonical parametrization.

We summarize the sections here. In the first section, we provide a canonical characterization of valuations of minimal generators in one-dimensional complete local domains via the Herzog--Kunz sequence. In the second section, we establish applications to torsion in differentials and conductor ideals, yielding more constructions which could help build evidence toward Berger’s conjecture. In the third section, we prove that arbitrary generating sets may be modified or truncated to recover the canonical parametrizations, thereby ensuring polynomial generation. 

\section{Main Result}

Let $k$ be a field, let $(R, \mathfrak m)$ be a complete 
local noetherian domain of dimension $1$, containing $k$ 
with $R/\mathfrak {m} = k$ and field of 
fractions $K = Q(R)$, and let $\overline{R}$ be its normalization. Then $\overline{R}$ 
is a complete regular local noetherian and equicharacteristic ring of dimension $1$. 
Denoting by $\overline{\mathfrak m}$ its maximal ideal,  we will assume that 
$\overline{R}/\overline{\mathfrak m} = k$ (which will be the case, if $k$ is 
algebraically closed), hence we may  write  $\overline{R} = k[[t]]$ for 
some local parameter $t$ of $\overline{R}$ and
	$$ R \subseteq \overline{R} $$ 
is a finite birational extension (\cite[Theorem 4.3.4]{SH}), hence has a parametrization  
	\begin{equation}\label{eq:rep_general}
	 R = k[[\alpha_1 t^{a_1}, \ldots, \alpha_n t^{a_n} ]] \subseteq k[[t]] 
	\end{equation}
with integers $1 \leq a_1\leq \cdots \leq a_n$ and units $\alpha_i \in \overline{R}$, where 
$n = \mathrm{edim}(R)$ is the embedding dimension of $R$. 

By $\mathfrak C_R$ we denote the conductor ideal of $R$ and by $c_R$ the conductor degree of $R$ and by $v$ we 
denote the $t$--valuation on $K = k((t))$. By $V(R)$ we 
will denote the valuation semigroup of $R$ (with respect 
to the valuation $v$), i.e. 
	$$ V(R)=\{v(f)\, \vert\,\, f \in R \setminus \{0\}\} $$
Then $V(R)$ is a numerical semigroup, i.e., it has a unique 
set of minimal generators 
	$$ 0 < w_1 < w_2 < \cdots < w_N $$
for some $N > 0$ with $\mathrm{gcd}\{w_1, \ldots, w_N\} 
= 1$, and, by~\cite[(2.4)]{HK}, the conductor 
degree $c_R$  of $R$ satisfies 
	\begin{enumerate}
	\item $c_R - 1 \notin V(R)$. 
	\item $c_R+i \in V(R)$ for all $i \geq 0$
	\end{enumerate}
which in particular implies that 
	$$ \mathfrak C_R = t^{c_R} \cdot \overline{R} = (t^{c_R}, 
    t^{c_R+1}, \ldots , t^{c_R+a_1-1}) \cdot R  $$
In general not much can be said about the structure of the 
exponents $a_i$ in parametrization~\eqref{eq:rep_general} 
other than that 
	$$ a_1 = \mathrm{min}\{v(f)\, \vert \,\, f \in 
    \mathfrak m \} = \mathrm{mult}(R) $$
is the multiplicity of $R$, but with some care the set 
$\{a_1, \ldots, a_n\}$ can be chosen in 
a unique  and more meaningful way:

\begin{satz}\label{thm:exist_unique}
There exist unique integers $1 \leq a_1 < a_2 < \cdots 
< a_n$ with the following properties
\begin{enumerate}
\item $n = \mathrm{edim}(R)$. 
\item If $x_1, \ldots, x_n \in R$ are any elements with 
$v(x_i) = a_i$ then 
	$$ R = k[[x_1, \ldots, x_n]] $$
\end{enumerate}
The set $\{a_1, \ldots, a_n\}$ is given by 
	$$ \{a_1, \ldots, a_n\} = v(\mathfrak m) \setminus 
    v(\mathfrak m^2) $$
\end{satz}

Crucial elements needed to prove this result are already 
contained in the treatise~\cite{HK} by Herzog and Kunz. 
One of their observations may be used to deduce the following.

\begin{lemma}\label{lem:frac_ideal_HK}
Let $I \subseteq K$ be a fractional ideal of $R$. Then the 
set 
	$$ S = v(I) \setminus v(\mathfrak m \cdot I) $$ 
is finite, say $S = \{b_1, \ldots, b_t\}$, and if 
$x_1, \ldots , x_t \in I$ are elements with $v(x_i)=b_i$, 
then $\{x_1, \ldots, x_t\}$ is a minimal set of generators 
of $I$.
\end{lemma}

\begin{proof}
For a fractional ideal $J \subseteq K$ we set 
(following~\cite{HK}, \S 2)
	$$ \mathrm{min}_J = \mathrm{min}\{ v(x) \, \vert \,\, x 
    \in J, x \neq 0 \} $$
Then by~\cite{HK}, 2.5. any $x \in \overline{R}$ with 
$v(x) \geq \mathrm{min}_{\mathfrak m \cdot I} + c_R$ 
satisfies $x \in \mathfrak m \cdot I$, implying the $S$ 
is finite. 

We write $S = \{b_1, \ldots, b_t\}$ with $b_1 < \cdots < b_n$ and choose any elments 
$x_1, \ldots , x_t \in I$ with $v(x_i) = b_i$. To show that 
$x_1, \ldots, x_t$ generate $I$,  it suffices to show, by 
Nakayama's lemma, that 
	$$ I = (x_1, \ldots, x_n) + \mathfrak m \cdot I $$
Assume to the contrary that there are $f \in I$ with $f 
\notin (x_1, \ldots, x_n) + \mathfrak m \cdot I $. Any such 
$f$ satisfies 
	$$ v(f) \leq \mathrm{max} \{0, \mathrm{min}_{\mathfrak m 
    \cdot I} + c_R - 1 \} $$ 
as otherwise $f \in \mathfrak m \cdot I$ by~\cite{HK}, 2.5. 
Thus 
	$$ a = \mathrm{max} \{ v(f) \, \vert \,\,  f \in I, f 
    \notin (x_1, \ldots, x_t) + \mathfrak m \cdot I  \} $$
is well defined. Choose $f \in I$, $f \notin  (x_1, \ldots, 
x_t) + \mathfrak m \cdot I$ with $v(f) = a$ and write 
$f = t^a \cdot g$ for some 
	\begin{equation}\label{eq:elem_not_in_1} 
	g = \rho_0 + \rho_1 \cdot t + \rho_2 \cdot t^2 + \cdots 
    \quad \in \overline{R} 
	\end{equation} 
with $\rho_j \in k$ and  $\rho_0 \neq 0$. 

If $a \in \{b_1, \ldots, b_t\}$, say $a = b_{i_0}$, then set 
$r = x_{i_0}$, if $a \notin \{b_1, \ldots, b_t\}$, then there 
exists an $r \in \mathfrak m \cdot I$ with 
$v(r) = a$. In both cases we may write 
	$$ r = t^a \cdot \left(\varepsilon_0 + \varepsilon_1 
    \cdot t + \varepsilon_2 \cdot t^2 + \cdots \right)  $$
with $\varepsilon_j \in k$ and  $\varepsilon_0 \neq 0$, and 
we set
	$$ h := \frac {\rho_0}{\varepsilon_0} \cdot r \, \in 
	(x_1, \ldots, x_t) + \mathfrak m \cdot I $$ 
to obtain $f-h \in I$, but $f -h \notin (x_1, \ldots, x_t) 
+ \mathfrak m \cdot I$ with $ v(f-h) > a$ 
by~\eqref{eq:elem_not_in_1}, contradicting the choice of $a$. 

Thus $I$ is generated by $x_1, \ldots, x_t$. To show that 
this set minimally generates $I$ it suffices to show that 
they are linearly independent modulo $\mathfrak m \cdot I$, 
and this follows from~\cite{HK} 2.9, implying that 
$\vert S \vert = \mathrm{dim}_k(I/\mathfrak m \cdot I)$ in 
our situation. 

For the convenience of the reader we include a short 
argument for this special case: 

Assume $x_1, \ldots , x_t$ are not linearly independent 
modulo $\mathfrak m \cdot I$ and denote by $\overline{x_i}$ 
the residue class of $x_i$ modulo $\mathfrak m \cdot I$. 
Then there exist $\varepsilon_1, \ldots, \varepsilon_t 
\in k$, not all $0$, with 
	\begin{equation}\label{eq:elem_not_lin_in_1} 
 	\varepsilon_1 \cdot \overline{x_1} + \cdots + 
    \varepsilon_t \cdot \overline{x_t} = 0 
	\end{equation}
Let $i_0 = \mathrm{min} \{ i \, \vert \,\, \varepsilon_i 
\neq 0 \}$. Then, as $k \subseteq R$, 
relation~\eqref{eq:elem_not_lin_in_1} gives a relation 
	\begin{equation}\label{eq:elem_not_lin_in_2} 
 	\varepsilon_{i_0} \cdot x_{i_0} + \cdots + \varepsilon_t 
    \cdot x_t = f
	\end{equation}
in $R$ for some $f \in \mathfrak m \cdot I$. But then 
	$$ b_{i_0} = v(x_{i_0}) = v(\varepsilon_{i_0} \cdot 
    x_{i_0} + \cdots + \varepsilon_n \cdot x_n) = v(f) 
    \in v(\mathfrak m \cdot I) $$
contradicting the choice of the $b_i$. 
\end{proof}

\begin{notiz}
Note that the set $S = \{b_1, \ldots, b_t\}$ constructed 
in~\ref{lem:frac_ideal_HK} is a subset of a minimal set 
$\Sigma = \{v_1, \ldots, v_T\}$ of generators of the value  
ideal $v(J)$ associated to $J$ in~\cite[\S 2]{HK}. In 
general however $S \subsetneq \Sigma$, thus 
Lemma~\ref{lem:frac_ideal_HK} strengthens~\cite[2.6]{HK}. 
\end{notiz}

\medbreak

\noindent\textit{Proof of Theorem~\ref{thm:exist_unique}.} 
From lemma~\ref{lem:frac_ideal_HK} we conclude that 
	$$ \{a_1, \ldots, a_n\} = v(\mathfrak m) \setminus 
    v(\mathfrak m^2) $$
is a finite set and that any $x_1, \ldots, x_n \in R$ with 
$v(x_i) = a_i$ (are contained in $\mathfrak m$ and) 
minimally generate $\mathfrak m$, hence 
	$$ R = k[[x_1, \ldots, x_n ]] $$ 
and $n = \mathrm{edim}(R)$. 

If $b_1 < \cdots < b_N$ is any other set of integers 
satisfying (1.) and (2.) of Theorem~\ref{thm:exist_unique}, 
then clearly $N = n \,\, (= \mathrm{edim}(R))$. 
If $\{b_1, \ldots , b_n\} \neq \{a_1, \ldots, a_n\}$ then 
for some $i$ necessarily $b_i \in v(\mathfrak m^2)$, hence 
we could choose $x_i \in \mathfrak m^2$ with 
$v(x_i) = b_i$, contradicting the fact that $x_1,\ldots,x_n$ 
minimally generate $\mathfrak m$ and completing the proof of 
Theorem~\ref{thm:exist_unique}. 

\medbreak

\begin{definition}
The sequence $a_1 < a_2 < \cdots < a_n$ is called 
\textbf{Herzog--Kunz sequence} of $R$, and elements 
$x_1, \ldots, x_n \in R$ with $v(x_i) = a_i$ are called 
\textbf{Herzog--Kunz generators} of $R$. 
\end{definition}

\begin{Corollary}\label{cor:mon_gen}
If for some $i_0 \in \{1, \ldots, n\}$, $a_{i_0} \geq c_R$, 
then we can choose $x_1, \ldots , x_n$ in such a way that 
$x_i = t^{a_i}$ for $i \geq i_0$. 

If $k$ is algebraically closed with $\mathrm{char}(k) = 0$, 
then the parameter $t$ of $\overline{R}$ can be chosen in 
such a way that in addition $x_1 = t^{a_1}$ is possible. 
\end{Corollary}

\begin{proof}
If $k$ is algebraically closed with $\mathrm{char}(k) = 0$, 
then by Hensel's Lemma (\cite[Theorem 7.3]{Ei}, see also 
\cite{HMM}, proof of Theorem 3.1) we may change the parameter 
$t$ of $\overline{R}$ in such a way that  $x_1 = t^{a_1}$. 
This doesn't change the valuation $v$, hence neither 
$a_1, \ldots, a_n$ nor $c_R$. 

For any $i \geq c_R$ we have that $t^i \in \mathfrak C_R 
\subseteq R$, hence for $i \geq i_0$ we may choose 
$x_i = t^{a_i}$. 
\end{proof}

The $a_i$ can also be found inductively. For this note first 
that necessarily 
	$$ a_1 =\mathrm{min}\{v(r)\,\vert \, r \in \mathfrak m 
    \setminus \{0\} \} = \mathrm{mult}(R) $$ 
is the multiplicity of $R$. Furthermore set 
	$$ T_i = k[[x_1,\ldots, x_i]] $$
($i = 1, \ldots, n-1$). 

\begin{Corollary}\label{cor:min_ai+1}
It holds
	$$ a_{i+1} = \mathrm{min} \{ w \,\vert \,\, w \in V(R) 
    \setminus V(T_i) \} $$
and for any $l \in \N$, if $i_0 = \mathrm{max} \{ i \, \vert 
\,\, a_i \leq l \}$, then any $f \in R$ may be written as 
	\begin{equation}\label{eq:split_val} 
	f = p + g  
	\end{equation}
with $p \in k[x_1, \ldots, x_{i_0}]$ and $g \in R$ with 
$v(g) \geq l+1$.
\end{Corollary}

\begin{proof}
Set 
	$$ b =  \mathrm{min} \{ w \,\vert \,\, w \in V(R) 
    \setminus V(T_i) \} $$
Clearly $a_{i+1} \leq b$ as every $f \in (x_{i+1}, \ldots, 
x_n)$ has valuation at least $a_{i+1}$. 

If $a_{i+1} < b$, then $a_{i+1} = v(f)$ for some 
$f \in \mathfrak m_{T_i}$, 
$f = \sum\limits_{j=1}^{i} r_j \cdot x_j$. As 
$a_{i+1} > a_i$, necessarily $r_j \in \mathfrak m_{T_i}$ 
for all $j$ (if some of the $r_j$ are units, we derive a 
contradiction as in the proof of the linear independence of 
$\overline{x_1}, \ldots, \overline{x_n}$ in 
$\mathfrak m/\mathfrak m^2$). But then 
	$$ a_{i+1} \in v(\mathfrak m_{T_i}^2) \subseteq 
    v(\mathfrak m^2) $$
a contradiction. 

For the second part we only need to consider $f \in R$ with 
$v(f) \leq l$. If we assume that not all such $f$ have a 
representation $f = p+g$ as desired, we may choose one of maximal 
valuation $b \, (\leq l)$. Then by the above, 
$b \in V(T_{i_0})$, hence there exists a $p \in T_{i_0}$ 
with $v(p) = b$. Clearly we may assume 
that $p \in k[x_1, \ldots, x_{i_0}]$. As in the proof of 
theorem~\ref{thm:exist_unique} 
we find an $\varepsilon \in k$ such that 
$v(f - \varepsilon \cdot p) > b$. But then 
$f - \varepsilon \cdot p$ would not have a representation as 
desired, contradicting the choice of $f$ and $b$. 
\end{proof}

\begin{notiz}
In equation~\eqref{eq:split_val} the $p \in k[x_1, \ldots, 
x_{i_0}]$ may be chosen to be a polynomial of degree at 
most $\lceil \frac {l}{a_1} \rceil$. 
\end{notiz}

\begin{Corollary}\label{cor:ai_value_sgr}
Assume the value semigroup $V(R)$ of $R$ is minimally 
generated by $w_1, \ldots, w_N$. 

Then 
	$$ a_1 = w_1, \quad a_2 = w_2 $$
and 
	\begin{equation}\label{eq:incl_ai_sgr} 
	\{a_1, \ldots, a_n\} \subseteq \{w_1, \ldots, w_N\} 
	\end{equation}
If $R$ is a numerical semigroup ring, then 
	$$ \{a_1, \ldots, a_n\} = \{w_1, \ldots, w_N\} $$
\end{Corollary}

\begin{notiz}
Corollary~\ref{cor:ai_value_sgr} shows that 
Theorem~\ref{thm:exist_unique} 
strengthens~\cite[3.6]{HK}.
\end{notiz}

\begin{beispiel}\label{ex:rep1}
Let $R = \C[[t^8,t^{12}+t^{14}+t^{15}]]$ (cf. 
\cite[p.3]{Fr}). Then $a_1 = 8$, $a_2 = 12$ and we may 
choose $x_1 = t^8$ and $x_2 = t^{12}+t^{14} + t^{15}$. 
Here we have $V(R) = \langle 8,12,26,55\rangle$, so this 
example shows that in general the 
inclusion~\eqref{eq:incl_ai_sgr} is strict and 
	$$ \mathrm{gcd}\{a_1, \ldots, a_n\}  > 
    \mathrm{gcd}\{w_1, \ldots, w_N\} \quad (\, = 1 \,) $$

Let $R = \C[[t^{6}, t^{9}+t^{10}, 2 t^{19}+t^{20}+t^{41}]]$. 
Then we have $V(R) = \langle 6, 9, 19, 41 \rangle$, but 
$a_1 = 6$, $a_2 = 9$, $a_3 = 41$. Here we may choose 
$x_1 = t^6$ and $x_2 = t^{9}+t^{10}$ , but 
	$$ x_2^2- x_1^3 = 2 t^{19}+ t^{20} $$
so $19 \in v(\mathfrak m^2)$. A direct computation shows 
that $41 \notin V(T_2)$.  

This example shows that the elements $a_1, \ldots, a_n$ 
are not necessarily the first $n$ elements of a minimal 
set of generators of $V(R)$. 
\end{beispiel}

\section{Torsion of Differential Forms}

As above let $k$ be an algebraically closed field, let 
$(R, \mathfrak m)$ be a complete local domain of dimension 
$1$, containing $k$ with $R/\mathfrak {m} = k$, let 
$\overline{R} = k[[t]]$ be its normalization and $c_R$ its 
conductor degree. Furthermore let $a_1 < \cdots < a_n$ be 
the Herzog--Kunz sequence of $R$. 

\begin{proposition}\label{prop:msquare}
If $f \in R$ with  $v(f)  > a_n$, then $f \in \mathfrak m^2$.
\end{proposition}

\begin{proof}
Let $f \in R$ be any element with $v(f) > a_n$. Then 
$v(f) \in v(\mathfrak m^2)$, and therefore, as in the proof 
of lemma~\ref{lem:frac_ideal_HK} we find an 
$h \in \mathfrak m^2$ with $v(f-h) > v(f)$. Repeating this 
a finite number of times, we obtain a $g \in \mathfrak m^2$ 
with $v(f-g ) \geq \mathrm{min}_{\mathfrak m^2} + c_R$ (with 
the notation introduced in the proof of 
lemma~\ref{lem:frac_ideal_HK}). Thus by~\cite[2.5]{HK}, 
$f-g \in \mathfrak m^2$, hence $f \in \mathfrak m^2$. 
\end{proof}

\begin{Corollary}\label{cor:cr_msquare} 
$\mathfrak C_R \not\subseteq \mathfrak m^2$ if and only if $a_n \geq c_R$.
\end{Corollary}

\begin{beispiel}\label{ex:cond_not_m2}
Let $R = \C[[t^{6}, t^{9}+t^{10}, 2 t^{19}+t^{20}+t^{41}]]$ 
with $V(R) = \{6,9,19,41\}$ as in example~\ref{ex:rep1}. 
Then the conductor degree $c_R = 36$ and the conductor ideal 
$\mathfrak C_R$ is not contained in $ \mathfrak m^2$ as it 
contains $x_3 = t^{41}$ which is not in $\mathfrak m^2$.  

This example shows that in Corollary~\ref{cor:cr_msquare} it 
is not sufficient to work with a minimal set of generators 
$x_1, \ldots, x_n$ such that $v(x_1), \ldots, v(x_n)$ 
is part of a minimal set of generators of $V(R)$.
\end{beispiel}

\begin{Corollary}\label{cor:cr_pure_power}
If $\mathrm{char}(k) = 0$ and $\mathfrak C_R \not\subseteq 
\mathfrak m^2$, then $R$ has a Herzog--Kunz parametrization 
$R = k[[x_1, \ldots, x_n]]$ with $x_1 = t^{a_1}$ 
and $x_n = t^{a_n}$, and 
	$$ \omega = a_n x_n dx_1 - a_1 x_1 dx_n \in 
    \widetilde{\Omega}^1_{R/k} $$
is a nonzero torsion element.  
\end{Corollary}

\begin{proof}
Use Corollary~\ref{cor:mon_gen}  and the proof of 
\cite[Theorem 3.1]{HMM}.
\end{proof}

\medbreak

Assume from now on that $\mathrm{char}(k) = 0$, $a_1, 
\ldots, a_n$ is the Herzog--Kunz sequence of $R$ and that 
$R$ is described by a parametrization 
	$$ R = k[[x_1, \ldots , x_n]] $$
with $v(x_i) = a_i$. By Hensel's Lemma (\cite[Theorem 7.3]{Ei}, 
see also \cite[Proof of Theorem 3.1]{HMM}) we may (and will 
always) assume in this situation that $x_1 = t^{a_1}$. 
For $i > 1$ write $x_i = \alpha_i \cdot t^{a_i}$ for some 
unit $\alpha_i \in R$, and write 
	$$ \alpha_i = \sum\limits_{j=0}^{\infty} u_{i,j} t^j $$
for suitable $u_{i,j} \in k$ and with $u_{i,0} \neq 0$. After 
multiplying $x_i$ with $\frac {1}{u_{i,0}}$, we may assume 
that $u_{i,0} = 1$. As in [MM], \S 4, we set 
	$$ \mathit{o}(\alpha_i) = v(\alpha_i - 1) = 
    \mathrm{min}\{j > 0 \,\vert \, u_{i,j} \neq 0 \} $$
if $\alpha_i-1 \neq 0$ and $\mathit{o}(\alpha_i) = \infty$ 
otherwise. 

\begin{proposition}\label{prop:MM_4_2_1}
Suppose there exist $i, d \geq 2$ such that $\mathit{o}
(\alpha_d) + a_i \geq c_R$ and $\mathit{o}(\alpha_d) < 
\infty$ and suppose $\mathit{o}(\alpha_d) \leq 
\mathit{o}(\alpha_i) \leq \infty$. Then $R$ has a 
parametrization 
	$$ R = k[[\widetilde{x_1}, \ldots, \widetilde{x_n}]] $$
with Herzog--Kunz generators $\widetilde{x_1}, \ldots, 
\widetilde{x_n}$ such that 
	$$ \begin{array} {l c l l}
	\widetilde{x_j} & = & x_j & \text{for }\, j \neq  i \\
	\widetilde{x_i} & = & t^{a_i} & 
	\end{array} $$ 
\end{proposition}

\begin{proof}
As noted in the proof of \cite[4.2 (1)]{MM}, the assumptions 
imply that 
	$$ \mathit{o}(\alpha_i) + a_i  \geq \mathit{o}(\alpha_d) 
    + a_i \geq c_R $$
hence 
	$$ x_i = \alpha_i \cdot t^{a_i} = t^{a_i} + b \cdot 
    t^{\mathit{o}(\alpha_i) +a_i} $$ 
with $b\cdot t^{\mathit{o}(\alpha_i) +a_i}\in \mathfrak C_R 
\subseteq R$. Thus 
	$$ \widetilde{x_i} = t^{a_i} = x_i - b \cdot 
    t^{\mathit{o}(\alpha_i) +a_i} \in R $$
with $v(\widetilde{x_i}) = a_i$, and therefore by 
theorem~\ref{thm:exist_unique} we may replace $x_i$ by 
$\widetilde{x_i}$. 
\end{proof}

\begin{proposition}\label{prop:MM_4_2_2}
Suppose there exist $i, d \geq 2$ such that 
$\mathit{o}(\alpha_d) + a_i \geq c_R$ and 
$\mathit{o}(\alpha_d) < \infty$ and suppose $i \leq d$. Then 
$R$ has a parametrization 
	$$ R = k[[\widetilde{x_1}, \ldots, \widetilde{x_n} ]] $$
with Herzog--Kunz generators $\widetilde{x_1}, \ldots, 
\widetilde{x_n}$ such that 
	$$ \begin{array} {l c l l}
	\widetilde{x_j} & = & x_j & \text{for }\, j \neq d \\
	\widetilde{x_d} & = & t^{a_d} & 
	\end{array} $$ 
\end{proposition}

\begin{proof}
As noted in the proof of \cite[4.2 (2)]{MM}, the assumptions 
imply that 
	$$ \mathit{o}(\alpha_d)+a_d \geq \mathit{o}(\alpha_d) + 
    a_i \geq c_R $$
hence again
	$$ x_d = \alpha_d \cdot t^{a_d} = 
	t^{a_d} + b \cdot t^{\mathit{o}(\alpha_d) +a_d} $$ 
with $b\cdot t^{\mathit{o}(\alpha_d) +a_d}\in \mathfrak C_R 
\subseteq R$, and we may continue as above. 
\end{proof}

\begin{notiz}
If in \cite[Theorem 4.2]{MM},  a parametrization 
	$$ R = k[[x_1, \ldots, x_n]] $$
with Herzog--Kunz generators $x_1, \ldots, x_n$ is chosen, 
then all the arguments work.
\end{notiz}

\bigbreak

Suppose now that $\mathfrak C_R \subseteq \mathfrak m^2$, 
let $c_R$ be the conductor degree of $R$ and set 
	$$ i_0 = \mathrm{max}\{i \in \{1, \ldots, n\}\,\vert \,
    \, a_i < c_R-a_1 \} $$
Furthermore set 
	$$ S = R[ \frac {\mathfrak C_R}{x_1}] $$
and let 
	$$ s = s(R) = 
	\mathrm{dim}_k \left( \frac {\left(\mathfrak C_R, 
    x_1\right)}{x_1} \right) $$ 
be the reduced type of $R$ \cite[\S 4]{HMM} (c.f. \cite{maitra2023extremal}). 

By \cite[(4.2)]{HMM} we have 
	$$ \mathfrak C_S = t^{c_R-a_1} \cdot \overline{R} = 
	\left(t^{c_R-a_1}, t^{c_R-a_1+1}, \ldots , t^{c_R-1} 
    \right) \cdot R $$
hence $c_S = c_R-a_1$, and there are $s$ elements  
$b_1 < b_2 < \cdots < b_s$ of the $c_R - a_1+i$ 
($i = 0,\ldots, a_1-1$) that are not in the 
value semigroup $V(R)$ of $R$. 

Then by \cite[\S 4]{HMM},  we have 
	$$ S = R[t^{b_1}, \ldots , t^{b_s}] $$
and $\mathrm{edim}(S) = n+s$ by \cite[Theorem 4.6]{HMM},  
implying that $x_1, \ldots, x_n, t^{b_1}, \ldots , t^{b_s}$
are minimally generating $\mathfrak m_S$, and the value 
semigroup 
	$$ V(S) = \langle V(R), b_1, \ldots, b_s \rangle $$
is generated by $V(R)$ and the $b_i$. Set 
	$$ i_0 = \mathrm{max}\{ i \, \vert a_i < c_R -a_1\} $$

\begin{proposition}\label{prop:MM_4_4_1}
In the above situation the Herzog--Kunz sequence 
$\widetilde{a_1} < \cdots < \widetilde{a_N}$ of $S$ 
satisfies 
	$$ \{\widetilde{a_{i_0+1}},\ldots,\widetilde{a_{n+s}}\} 
    = \{a_{i_0+1}, \ldots, a_{n}, b_1,\ldots, b_s \} $$
and $S$ has a parametrization 
	$$ S = k[[\widetilde{x_1},\ldots,\widetilde{x_{n+s}}]] $$
with Herzog--Kunz generators $\widetilde{x_1},\ldots,
\widetilde{x_{n+s}}$  such that 
	$$ \begin{array} {l c l l}
	\widetilde{x_j} & = & x_j & \text{for }\, j \leq i_0 \\
	\widetilde{x_j} & = & t^{\widetilde{a_j}} & \text{for } 
    \, j > i_0 \\
	\end{array} $$  
\end{proposition}

\begin{proof}
The Herzog--Kunz sequence $\widetilde{a_1} < \cdots < 
\widetilde{a_N}$ of $S$ satifies $N = n+s$ as 
$\mathrm{edim}(S) = n+s$. Thus it suffices to show that 
	$$ \{ \widetilde{a_1} ,\ldots ,\widetilde{a_{n+s}}\} = 
	\{a_1, \ldots, a_n, b_1, \ldots b_s\} $$
for then for $j > i_0$, $\widetilde{a_j} \geq c_S$, and 
therefore $t^{\widetilde{a_j}} \in S$.

As $\mathfrak m_S = (x_1, \ldots, x_n, t^{b_1}, \ldots, 
t^{b_s})$, any $s \in \mathfrak m$ may be written as 
	$$ s = r + \sigma $$
with $r \in \mathfrak m$ and $v(\sigma) \geq b_1 $, and 
similarly any $s_2 \in \mathfrak m_S^2 = (x_i \cdot x_j, 
x_i \cdot t^{b_l}, t^{b_l} \cdot t^{b_m})$ 
may be written as 
	$$ s_2 = r_2 + \sigma_2 $$
with $r_2 \in \mathfrak m^2$ and $v(\sigma_2) \geq b_1 \,
(\geq c_R -a_1)$. Thus 
	$$ V_1 = \{w\in v(\mathfrak m)\,\vert\,w < c_R -a_1\} = 
	\{w\in v(\mathfrak m_S)\,\vert \, w < c_R -a_1\} = W_1 $$
and 
	$$ V_2=\{ w\in v(\mathfrak m^2)\,\vert \,w<c_R-a_1\} = 
	\{w \in v(\mathfrak m^2_S)\,\vert \,w < c_R -a_1\}=W_2 $$
implying that 
	$$ \{ \widetilde{a_i} \, \vert \, \widetilde{a_i} 
    \leq c_R-a_1 \} = W_1 \setminus W_2
	= V_1 \setminus V_2 =\{a_i\, \vert \,a_i\leq c_R-a_1 \} 
    = \{a_1, \ldots, a_{i_0} \} $$
It remains to show that 
	$$\{\widetilde{a_{i_0+1}},\ldots,\widetilde{a_{n+s}}\}= 
	\{ a_{i_0+1}, \ldots, a_n, b_1, \ldots, b_s \} $$
As both sets have the same cardinality, it suffices to show 
the inclusion $\subseteq$. 

Note that $\widetilde{a_{n+s}} < c_R$ as otherwise 
	$$ \widetilde{a_{n+s}} \in v(\mathfrak C_R) \subseteq 
    v(\mathfrak m^2) \subseteq v(\mathfrak m_S^2) $$
and similarly $a_n < c_R$. Thus both sets are contained in 
$\{c_R-a_1, \ldots, c_R-1\}$. 

If $u \in \{w \in V(S) \, \vert \, c_R-a_1 \leq w < c_R \} 
\setminus \{ a_{i_0+1}, \ldots, a_n, b_1, \ldots, b_s\}$, 
then $u \in V(R)$ (by the choice of $b_1, \ldots, b_n$) 
and thus 
$u \in v(\mathfrak m^2)$ (by the choice of $a_{i_0+1}, 
\ldots, a_n$), hence $u \in v(\mathfrak m_S^2)$, 
and therefore 
	$$ u \notin \{ \widetilde{a_{i_0+1}}, \ldots, 
    \widetilde{a_{n+s}} \} $$
completing the proof. 
 
\end{proof}

\begin{Corollary}
If in the above situation $a_1 + a_n \geq c_R$, then $S$ 
has a parametrization 
	$$ S=k[[\widetilde{x_1},\ldots,\widetilde{x_{n+s}}]] $$
with Herzog--Kunz generators $\widetilde{x_1}, \ldots, 
\widetilde{x_{n+s}}$ such that $t^{a_1}, 
t^{a_n}, t^{b_1}, \ldots t^{b_n}$ are among the  
$\widetilde{x_i}$. 

Hence if in \cite[Theorem 4.4]{MM}, a parametrization  
	$$ R = k[[x_1, \ldots, x_n]] $$
with Herzog-Kunz generators $\widetilde{x_1}, \ldots, 
\widetilde{x_{n+s}}$ is chosen, then all the arguments work.
\end{Corollary}

\begin{proof}
First not that $\widetilde{x_1} = x_1 = t^{a_1}$ by our 
assumption in this section. Furthermore in this case 
$i_0 < c_R -a_1 \leq a_n$, thus 
Proposition~\ref{prop:MM_4_4_1} completes the proof. 
\end{proof}

\section{Changing Parametrizations}

Theorem~\ref{thm:exist_unique} says that, once a 
Herzog--Kunz sequence $(a_1, \ldots, a_n)$ for the ring 
$R$ is determined, any set of elements $x_1, \ldots , x_n 
\in R$ with $v(x_i) = a_i$ will (minimally) generate $R$. 
Thus any such sequence may be modified quite arbitrarily and 
will stay a set of generators as long as $v(x_i) = a_i$. For 
a general set of generators such a result is not true, but 
in this section it will be shown that any set $y_1, \ldots, 
y_n$ of generators can be modified in sufficiently high 
degrees and will stay a set of generators. 

Throughout this section we consider the usual situation 
	$$ R = k[[y_1, \ldots , y_n]] \subseteq \overline{R} 
        = k[[t]] $$
and we assume that $a_1 < \cdots < a_n$ is the Herzog--Kunz 
sequence associated to $R$. We assume that $y_1, \ldots, 
y_n$ is minimally generating $\mathfrak m$, but we do not 
assume that $v(y_i) = a_i$. 

Our first result in this section shows that the Herzog--Kunz 
sequence $a_1, \ldots, a_n$ nevertheless is present in 
$y_1, \ldots, y_n$. 

\begin{proposition}\label{prop:optimal_part}
After possibly reordering the elements $y_1, \ldots, y_n$ 
there exist elements $x_1, \ldots, x_n \in R$ and 
$z_1, \ldots, z_n \in R$ with the following properties 
\begin{enumerate}
\item $v(x_i) = a_i$. 
\item $z_i \in k[x_1, \ldots, x_{i-1}]$ (with $z_1 = 0$).
\item $y_i = x_i + z_i$. 
\end{enumerate}
 for all $i = 1, \ldots , n$
\end{proposition}

\begin{proof}
We may assume that $v(y_1) = a_1$ and we set 
	$$ x_1 = x_1^{(1)} = y_1, \quad z_1 = z_1^{(1)} = 0 $$
Inductively we will construct for $k = 2, \ldots n$ elements 
$x_k^{(k)}, \ldots, x_n^{(k)}$and $z_k^{(k)}, \ldots, 
z_n^{(k)}$ with the following properties 
\begin{enumerate}
\item $v(x_i^{(k)}) \geq a_k$ for $i =k, \ldots, n$. 
\item $v(x_k^{(k)}) = a_k$. 
\item $z_i^{(k)} \in k[x_1, \ldots, x_{k-1}] $ for  
$i = k, \ldots n$. 
\item After possibly reordering $y_k, \ldots, y_n$, 
$y_i = x_i^{(k)} + z_i^{(k)}$ for  $i = k, \ldots n$.
\end{enumerate}
Then setting $x_i = x_i^{(i)}$ and $z_i = z_i^{(i)}$ for 
$i = 2, \ldots, n$ will complete the proof. 

Case $k = 2$: By Corollary~\ref{cor:min_ai+1} we can write 
	$$ y_i = p_i + g_i $$
for some $p_i \in k[x_1^{(1)}]$ and some $g_i \in R$ with  
$v(g_i) \geq a_2$ for each $i = 2, \ldots, n$. Then 
	$$ k[[y_1, \ldots, y_n]] = k[[x_1^{(1)} , g_2, \ldots, g_n]] $$
and therefore for at least one $i \in \{2, \ldots, n\}$ we 
must have $v(g_i) = a_2$. We may assume this to be $g_2$, 
and we set 
	$$ z_i^{(2)} = p_i, \quad 
	x_i^{(2)} = g_i $$

For the inductive step we note that the inductive hypothesis 
implies 
	$$ R = k[[x_1^{(1)}, \ldots, x_{k-1}^{(k-1)}, y_k, 
    \ldots, y_n]] $$
Using Corollary~\ref{cor:min_ai+1} again, we get 
	$$ y_i = p_i + g_i $$
with $p_i \in k[x_1^{(1)}, \ldots, x_{k-1}^{(k-1)}]$ and 
$g_i \in R$ with $v(g_i) \geq a_i$, and we may proceed in 
exactly the same way as in step $k = 2$. 
\end{proof}

\begin{notiz} 
Proposition~\ref{prop:optimal_part} allows to construct the 
Herzog--Kunz sequence $(a_1, \ldots, a_n)$, starting with 
an arbitrary parametrization $R = k[[y_1, \ldots, y_n]] 
\subseteq k[[t]]$ of $R$: 

If $R = k[[t^6, t^9+t^{10}, 2 t^{19} +t^{20}+t^{41}]]$ is 
as in Example~\ref{ex:rep1}, then $y_1=t^6, 
y_2=t^{9}+t^{10}$ and $y_3 = 2t^{19}+t^{20}+t^{41}$ 
minimally generate $R$. Obviously $a_1=6$, as this is the 
smallest positive valuation in $R$ and $a_2=9$ as the value 
$9$ cannot be obtained with $y_1$ alone. However as $3\cdot 
a_1 = 2 \cdot a_2$, $y_2^2-y_1^3$ provides a relation to 
modify $y_3$, and in fact $x_3 = y_3-(x_2^2-y_1^3) = 
t^{41}$, hence $a_3 = 41$, and the construction is complete. 

As a byproduct we obtain that the numerical semigroup 
$V(R)$ is minimally generated by $6,9,19$ and $41$.
\end{notiz}

\begin{satz}\label{thm:rep_chg}
Let $R = k[[y_1, \ldots, y_n]]$ as above with $y_1, \ldots, 
y_n$ minimally generating $\mathfrak m$. Furthermore let 
$\widetilde{y_1}, \ldots, \widetilde{y_n} \in R$ be elements 
with 
	$$ y_i = \widetilde{y_i} \pmod{t^{a_n+1}} \quad \text{ 
    for all } i = 1, \ldots, n $$
Then 
	$$ R = k[[\widetilde{y_1}, \ldots, \widetilde{y_n} ]] $$
\end{satz}

\begin{proof}
It suffices to show that $\widetilde{y_1}, \ldots, 
\widetilde{y_n}$ generate $\mathfrak m$. But as 
$v(y_i -\widetilde{y_i}) > a_n$ we have by 
proposition~\ref{prop:msquare} that $y_i - \widetilde{y_i} 
\in \mathfrak m^2$, hence
	$$ (y_1, \ldots, y_n) + \mathfrak m^2 = 
    (\widetilde{y_1}, \ldots, \widetilde{y_n}) 
	+ \mathfrak m^2 $$
and therefore by the Nakayama--Lemma $\widetilde{y_1}, 
\ldots, \widetilde{y_n}$ generate $\mathfrak m$. 
\end{proof}

\begin{Corollary}\label{cor:rep_trunc}
Let $R = k[[y_1, \ldots, y_n]] \subseteq \overline{R} 
= k[[t]]$ as above with $y_1, \ldots, y_n$ 
minimally generating $\mathfrak m$, let $d = 
\mathrm{max}\{a_n+1,c_R\}$ and let 
$\widetilde{y_i}$ be the truncation of $y_i$ at $t^d$. Then 
	$$ R =  k[[\widetilde{y_1}, \ldots, \widetilde{y_n} ]] $$
In particular $R$ can always be generated by polynomials in 
$t$.
\end{Corollary}

\begin{notiz}
There are various results available in the literature of a 
flavor similar to Theorem~\ref{thm:rep_chg} resp., 
Corollary~\ref{cor:rep_trunc} (cf. \cite[(1.1)]{CC}, 
or \cite[(2.1)]{Ca}). Note however that (for truncation)
$d = \mathrm{max}\{a_n+1, c_R\}$ is the best bound possible, 
and in particular $c_R$ is not valid in general. For this 
consider 
	$$ R = k[[t^5,t^7,t^{12}+t^{23}]] $$
This ring has value semigroup $V(R) = <5,7,23>$ and its 
Herzog--Kunz sequence is $a_1 = 5, a_2 = 7$ and $a_3 = 23$. 
Furthermore $c_R = 19$, but 
	$$ R' = k[[t^5, t^7,t^{12}]] \subsetneq R $$ 
The reason for this is that $V(R') = <5,7>\,$ with conductor 
degree $c_{R'} = 24$, so $t^{23}$ cannot be reinserted 
anymore.  The conductor degree of the resulting 
curve singularity might change by truncating the generators 
in degree $t^{c_R}$.
\end{notiz}

\begin{notiz}
The bound $d$ in Corollary~\ref{cor:rep_trunc} cannot be 
calculated explicitely without determining the value 
semigroup of $R$ first. However any bound $c^{\star}$ 
for the conductor $c_R$ of $R$ will give the bound 
	$$ d^{\star} = c^{\star} + \mathrm{mult}(R) - 1 $$
where $\mathrm{mult}(R) = a_1$ is the multiplicity of $R$, 
as necessarily 
	$$ a_n \leq c_R + \mathrm{mult}(R) - 1 $$ 
In \cite[\S 1]{CC}  Castellanos and Castellanos provide a 
bound $c^{\star}$ for $c_R$ in terms of the multiplicity 
sequence of $R$. 
\end{notiz}

\bibliographystyle{plain}
\bibliography{bibliography}
\end{document}